\documentclass[a4paper,12pt,leqno]{article}
\usepackage[margin=3cm]{geometry}
\usepackage{mathtools}
\usepackage{tikz-cd}
\usepackage{amssymb}
\usepackage{amsxtra}
\usepackage{amsthm}
\usepackage[T1]{fontenc}
\usepackage{lmodern}
\usepackage[pagebackref,breaklinks]{hyperref}
\usepackage[lite,abbrev,shortalphabetic]{amsrefs}
\usepackage[shortlabels]{enumitem}
\usepackage{colonequals}
\usepackage{tensor}
\theoremstyle{definition}
\newtheorem{remark}[subsection]{Remark}
\newtheorem{construction}[subsection]{Construction}
\newtheorem{question}[subsection]{Question}
\newtheorem{example}[subsection]{Example}
\newtheorem{defn}[subsection]{Definition}
\theoremstyle{plain}
\newtheorem{lemma}[subsection]{Lemma}
\newtheorem{prop}[subsection]{Proposition}
\newtheorem{cor}[subsection]{Corollary}
\newtheorem{theorem}[subsection]{Theorem}
\newtheorem{conj}[subsection]{Conjecture}
\numberwithin{equation}{subsection}
\setlist[enumerate,1]{(a)}

\newcommand{\Z}{\mathbb{Z}}
\newcommand{\Q}{\mathbb{Q}}

\newcommand{\F}{\mathbb{F}}
\newcommand{\C}{\mathbb{C}}

\newcommand{\cA}{\mathcal{A}}

\newcommand{\cD}{\mathcal{D}}

\newcommand{\cF}{\mathcal{F}}
\newcommand{\cG}{\mathcal{G}}
\newcommand{\cH}{\mathcal{H}}
\newcommand{\cI}{\mathcal{I}}
\newcommand{\cJ}{\mathcal{J}}

\newcommand{\cL}{\mathcal{L}}

\newcommand{\cO}{\mathcal{O}}
\newcommand{\cP}{\mathcal{P}}

\newcommand{\cX}{\mathcal{X}}
\newcommand{\cY}{\mathcal{Y}}
\newcommand{\cZ}{\mathcal{Z}}
\newcommand{\cHom}{\mathcal{H}\mathit{om}}

\newcommand{\Perv}{\mathrm{Perv}}
\newcommand{\Hom}{\mathrm{Hom}}
\newcommand{\End}{\mathrm{End}}
\newcommand{\Shv}{\mathrm{Shv}}

\newcommand{\Ext}{\mathrm{Ext}}
\newcommand{\id}{\mathrm{id}}
\newcommand{\Id}{\mathrm{id}}

\newcommand{\cl}{\mathrm{cl}}
\newcommand{\diag}{\mathrm{diag}}

\newcommand{\Spec}{\mathrm{Spec}}
\newcommand{\Spf}{\mathrm{Spf}}
\newcommand{\gr}{\mathrm{gr}}

\newcommand{\chara}{\mathrm{char}}

\newcommand{\red}{\mathrm{red}}
\newcommand{\alg}{\mathrm{alg}}
\newcommand{\Fr}{\mathrm{Fr}}
\newcommand{\Gal}{\mathrm{Gal}}
\newcommand{\IH}{\mathrm{IH}}
\newcommand{\IC}{\mathrm{IC}}

\newcommand{\Qlb}{\overline{\Q_\ell}}

\newcommand{\simto}{\xrightarrow{\sim}}

\newcommand{\pH}{\prescript{p}{}{\cH}}
\newcommand{\ptau}{\prescript{p}{}{\tau}}

\begin{document}
\title{Slopes and weights of $\ell$-adic cohomology of rigid spaces}
\author{Qing Lu\thanks{School of Mathematical Sciences, Beijing Normal University, Beijing
100875, China; email: \texttt{qlu@bnu.edu.cn}.} \and Weizhe
Zheng\thanks{Morningside Center of Mathematics, Academy of Mathematics and
Systems Science, Chinese Academy of Sciences, Beijing 100190, China;
University of the Chinese Academy of Sciences, Beijing 100049, China; email:
\texttt{wzheng@math.ac.cn}.}}\date{} \maketitle

\begin{abstract}
We prove that Frobenius eigenvalues of $\ell$-adic cohomology and 
$\ell$-adic intersection cohomology of rigid spaces over $p$-adic local 
fields are algebraic integers and we give bounds for their $p$-adic 
valuations. As an application, we deduce bounds for their weights, proving 
conjectures of Bhatt, Hansen, and Zavyalov. We also give examples of  
monodromy-pure perverse sheaves on projective curves with non 
monodromy-pure cohomology, answering a question of Hansen and Zavyalov. 
\end{abstract}

\section{Introduction}

Let $K$ be a $p$-adic local field of finite residue field $\F_q$ and let 
$\ell\neq p$ be a prime. Arithmetic properties of $\ell$-adic cohomology of 
rigid spaces over $K$, including the $\ell$-adic intersection cohomology 
defined by Bhatt and Hansen \cite{BH}, were recently studied by Hansen and 
Zavyalov \cite{HZ}. In this paper we study further arithmetic properties, 
including integrality, $p$-adic valuations, and weights of Frobenius 
eigenvalues. 

Before stating our results let us fix some notation and convention. For an 
algebraic number $\alpha$, we call \emph{$q$-slopes} of $\alpha$ the rational 
numbers $v_q(\iota(\alpha))$, where $v_q$ is the valuation on $\C_p$ 
normalized by $v_q(q)=1$ and $\iota$ runs through all embeddings $\iota\colon 
\Q(\alpha)\to \C_p$. Let $G_K\to G_{\F_q}$ be the canonical projection of 
Galois groups and let $\Fr_{q}\in G_{\F_q}$ be the geometric Frobenius. Let 
$C$ be a completed algebraic closure of $K$. For a quasi-compact 
quasi-separated (qcqs) rigid space $X$ over $K$, we put 
$\IH^i_{(c)}(X_C,\Q_\ell)=H^{i}_{(c)}(X_C,\IC_{X,\Q_\ell})$ as in \cite{BH} 
and \cite{HZ}, where $\IC_{X,\Q_\ell}$ denotes the intersection complex of 
$X$. 

Our main result on slopes is the following. 

\begin{theorem}\label{t.main}
Let $X$ be a qcqs rigid space of dimension $d$ over $K$ and let $g\in G_K$ be 
an element projecting to $\Fr_{q}$. 
\begin{enumerate}
\item For all $i\in \Z$, the eigenvalues of $g$ acting on 
    $\IH^{i-d}(X_C,\Q_\ell)$, $\IH^{i-d}_c(X_C,\Q_\ell)$, and 
    $H^i_c(X_C,\Q_\ell)$ are algebraic integers that are units away from 
    $p$ and their $q$-slopes are contained in $[0,i]\cap [i-d,d]$. 
\item For all $i\in \Z$, the eigenvalues of $g$ acting on 
    $H^i(X_C,\Q_\ell)$ are algebraic integers that are units away from $p$ 
    and their $q$-slopes are contained in $[0,i]$. 
\end{enumerate} 
\end{theorem}

R.~Huber \cite{Huber} defined compactly supported cohomology of separated 
rigid spaces. In the nonseparated case one can define compactly supported 
cohomology using duality: 
\[H^i_c(X_C,\Q_\ell)=H^{-i}(X_C,D_{X}\Q_\ell)\spcheck,\quad \IH^i_c(X_C,\Q_\ell)=\IH^{-i}(X_C,D_{X}\IC_{X,\Q_\ell})\spcheck,\]
where $D_{X}(-)=R\cHom(-,\omega_{X})$ is Verdier duality \cite[Theorems 3.21, 
3.36]{BH}. 

In case (a), the $q$-slopes of the eigenvalues of $g$ lie in the shaded 
region of the following picture. 
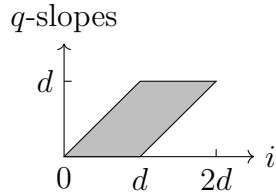
\begin{center}
\begin{tikzpicture}
    \draw[->] (0,0) -- (2.5,0) node[right] {$i$};
    \draw (1,0.1) -- (1,0) node[below] {$d$};
    \draw (2,0.1) -- (2,0) node[below] {$2d$};
    \draw (0,0) node[below] {$0$};
    \draw[->] (0,0) -- (0,1.5) node[above] {$q$-slopes};
    \draw (0.1,1) -- (0,1) node[left] {$d$};
    \draw[fill=lightgray] (0,0) -- (1,0) -- (2,1) -- (1,1) -- cycle;
\end{tikzpicture}
\end{center}

\begin{remark}
\begin{enumerate}
\item That the eigenvalues of $g$ acting on $H^i_c(X_C,\Q_\ell)$ are 
    algebraic integers is a result of Mieda \cite[Theorem 1.1]{Mieda}.
\item In the case where $X$ is smooth, the results of Theorem \ref{t.main} 
    hold also over a local field of characteristic $p>0$. See Proposition 
    \ref{p.pos}. 
\item The analogue of Theorem \ref{t.main} for schemes of finite type over 
    a local field holds with $[0,i]\cap [i-d,d]$ as bounds for $q$-slopes 
    in all four cases. See Remark \ref{r.sch}. 
\end{enumerate}
\end{remark}

Hansen and Zavyalov showed that the eigenvalues are Weil numbers 
\cite[Theorem 1.2.5]{HZ}. Combining their theorem with ours, one gets the 
following. 

\begin{cor}\label{c.wt}
Let $X$ be a qcqs rigid space of dimension $d$ over $K$ and let $g\in G_K$ be 
an element projecting to $\Fr_{q}$. 
\begin{enumerate}
\item For all $i\in \Z$, the eigenvalues of $g$ acting on 
    $\IH^{i-d}(X_C,\Q_\ell)$, $\IH^{i-d}_c(X_C,\Q_\ell)$, and
    $H^i_c(X_C,\Q_\ell)$ are $q$-Weil numbers with weights contained in 
    $[0,2i]\cap [2(i-d),2d]$. 
\item For all $i\in \Z$, the eigenvalues of $g$ acting on 
    $H^i(X_C,\Q_\ell)$ are $q$-Weil numbers with weights contained in 
    $[0,2i]$. 
\end{enumerate} 
\end{cor}

\begin{remark}
For $\IH^*$ the above corollary proves a conjecture of Bhatt and Hansen 
(\cite[Conjecture 4.15(2)]{BH}, \cite[Conjecture 5.1(2)]{HZ}). For $\IH_c^*$ 
the corollary proves a conjecture of Hansen and Zavyalov \cite[Conjecture 
5.1(1)]{HZ}. Some weaker bounds of weights were already proven in 
\cite[Theorem 3.3.4]{HZ}. The idea of using slope bounds to improve weight 
bounds originated from Deligne \cite[3.3.2]{WeilII}. 
\end{remark}

Recall that the action of the inertia group $I_K$ of $K$ is quasi-unipotent 
by Grothendieck's monodromy theorem \cite[Appendix]{ST}. Hansen and Zavyalov 
showed that the quasi-unipotency index can be bounded uniformly and, in the 
case of $H^i$, the index is bounded by $i$ \cite[Theorems 1.3.1, 1.3.2]{HZ}. 
Our slope bounds imply explicit bounds for the quasi-unipotency index in all 
four cases. 

\begin{cor}\label{c.qu}
Let $X$ be a qcqs rigid space of dimension $d$ over $K$. Then there exists an 
open subgroup $I_1\subseteq I_K$, independent of $\ell$, such that for all 
$g\in I_1$ and $i\in [0,2d]$, we have 
\begin{enumerate}
\item $(g-1)^{i+1}=0$ on $\IH^{i-d}_{(c)}(X_C,\Q_\ell)$ and 
    $H^i_{(c)}(X_C,\Q_\ell)$,
\item $(g-1)^{2d-i+1}=0$ on $\IH^{i-d}_{(c)}(X_C,\Q_\ell)$ and 
    $H^i_{c}(X_C,\Q_\ell)$.
\end{enumerate} 
\end{cor}

\begin{remark}
The case $H^i$ of part (a) was already proved in \cite[Theorem 1.3.2]{HZ} in 
a different way. The other three cases of part (a) prove the local field case 
of \cite[Conjecture 5.3]{HZ} with $\Q_\ell$-coefficients. In a forthcoming 
paper, we consider the more general case of $p$-adic complete discrete 
valuation fields. The analogue of part (a) for separated schemes of finite 
type over a complete discrete valuation field was proved by Gabber and 
Illusie \cite[Theorem 2.3, Remark 2.5]{IllTou}. 
\end{remark}

We deduce Theorem~\ref{t.main} from a result on the integrality and slopes of 
nearby cycles. Let $\cO_K$ be the ring of integers of $K$. By an 
\emph{admissible} formal $\cO_K$-scheme we mean a flat formal $\cO_K$-scheme 
topologically of finite type. Recall from \cite[Definition A.3.1]{HZ} that 
for any admissible formal $\cO_K$-scheme $\cX$, the nearby cycles live on the 
fiber product topos $\cX_s\times_s\eta$, where $s=\Spec(\F_q)$, 
$\eta=\Spec(K)$, and $\cX_s$ denotes the special fiber of $\cX$. 

\begin{theorem}\label{t.psi}
Let $\cX$ be an admissible formal $\cO_K$-scheme with generic fiber 
$\cX_\eta$ of dimension $d$. Then, for any $i\in \Z$, any geometric point 
$\bar x$ above $x$ of $\cX_s$, and any element $g$ of $G_x\times_{G_{\F_q}} 
G_K$ projecting to the geometric Frobenius $\Fr_x\in G_x$,  the eigenvalues 
of $g$ acting on $(R^i\Psi_{\cX}\Q_\ell)_{\bar x}$ and 
$(R^{i-d}\Psi_{\cX}\IC_{\cX_\eta,\Q_\ell})_{\bar x}$ are algebraic integers 
that are units away from $p$ and their $q_x$-slopes are contained in $[0,i]$. 
Here $q_x$ denotes the cardinality of the residue field of $x$. 
\end{theorem}

The proof of the theorem relies on the study of integral sheaves and 
complexes on schemes, which was initiated by Deligne \cite[XXI 5]{SGA7II} and 
developed in \cite{Zint}. One key new ingredient for the case of intersection 
complexes is the preservation of integral complexes under perverse 
truncations and perverse subquotients. The preservation is proved in Section 
\ref{s.2} in a general setting, following ideas from weight theory of 
perverse sheaves \cite[Section~5]{BBD}. The proofs of Theorems \ref{t.main} 
and \ref{t.psi} are given in Section \ref{s.3}.

We now turn to the monodromy weight property. For $Y$ a scheme of finite type 
over $\F_q$, recall from \cite[Definition 2.7.1]{HZ} that $L\in 
D^b_c(Y\times_s \eta,\Q_\ell)$ is said to be \emph{monodromy-pure} of weight 
$w$ if $\gr^M_j\pH^i L$ is pure of weight $w+i+j$ for all $i,j\in \Z$, where 
$M$ is the monodromy filtration. Hansen and Zavyalov \cite{HZ}*{Question 
1.4.9} asked the following question. 

\begin{question}\label{q.monodromy}
Let $f\colon Y\to Z$ be a projective morphism of schemes of finite type over 
$\F_q$. Does $R(f\times_s \eta)_*$ preserve monodromy-pure complexes of 
weight $w$? 
\end{question}

A positive answer to the question would imply both Deligne's monodromy weight 
conjecture for schemes and its variant for rigid spaces \cite[Conjecture 
1.4.7]{HZ}. We answer the question negatively by constructing  
counterexamples in Section~\ref{s.4}. 

\begin{theorem}\label{t.ex}
For every projective curve $Y$ over $\F_q$, there exists a monodromy-pure 
perverse sheaf $\cP\in \Perv_c(Y\times_s \eta,\Q_\ell)$ such that 
$R(a\times_s \eta)_*\cP$ is not monodromy-pure. Here $a\colon Y\to s$ is the 
structure morphism. 
\end{theorem}

We end this introduction with a speculation. The bounds for slopes, weights, 
and the quasi-unipotency index we gave for $\IH_{(c)}^{i-d}(X_C,\Q_\ell)$ and 
$H_c^i(X_C,\Q_\ell)$ are optimal, as shown by the example of the rigid 
analytification of products of elliptic curves with multiplicative reduction. 
The same example shows that the bounds we gave for $H^i(X_C,\Q_\ell)$ are 
also optimal for $i\le d$. It is natural to expect that $H^i(X_C,\Q_\ell)$ 
admits the same bounds as $H^i_c(X_C,\Q_\ell)$ for every~$i$. 

\begin{conj}
Let $X$ be a qcqs rigid space of dimension $d$ over $K$ and let $g\in G_K$ be 
an element projecting to $\Fr_{q}$. For all $i>d$, the $q$-slopes of the 
eigenvalues of $g$ acting on $H^i(X_C,\Q_\ell)$ are contained in $[i-d,d]$. 
\end{conj}

\subsection*{Conventions} By a \emph{local field} we mean a complete discrete 
valuation field with finite residue field. For a field $k$, we denote by 
$\bar k$ its algebraic closure and by $G_k=\Gal(\bar k/k)$ its Galois group. 

\subsection*{Acknowledgments} We learned much about the cohomology of rigid 
spaces from the work of Bhatt, Hansen, and Zavyalov (\cite{BH}, \cite{HZ}). 
We thank Shizhang Li and Xiangdong Wu for useful conversations and comments 
on drafts of this paper. We thank the referee for many helpful suggestions.

This work was partially supported by National Key Research and Development 
Program of China (grant number 2020YFA0712600), National Natural Science 
Foundation of China (grant numbers 12125107, 12271037, 12288201), Chinese 
Academy of Sciences Project for Young Scientists in Basic Research (grant 
number YSBR-033). 

\section{Perverse truncations and usual truncations}\label{s.2}

In this section, we show that integral complexes on schemes are stable under 
perverse truncations and integral perverse sheaves are stable under 
subquotients. The proof is inspired from weight theory of perverse sheaves 
\cite[Section~5]{BBD}. To showcase the similarities between the two scenarios 
and for other applications, we work in a general setting. 

Let $\cD$ be a triangulated category. Recall that a full subcategory 
$\cJ\subseteq \cD$ is said to be \emph{stable under extensions} if for every 
distinguished triangle 
\[L'\to L\to L''\to L'[1]\] 
with $L',L''\in \cJ$, we have $L\in \cJ$. We are interested in full 
subcategories $\cJ\subseteq \cD$  stable under extensions and satisfying 
$\cJ\subseteq \cJ[1]$. In this case, if $L,L''\in \cJ$ in the above 
distinguished triangle, then $L'\in \cJ$. 

Given $\cJ$ as above and a t-structure on $\cD$ with heart $\cA$, we put 
$\cA_\cJ^{[i]}=\cA\cap (\cJ[i])$ for all $i\in \Z$. By assumption 
$\cA_\cJ^{[i]}\subseteq \cA_\cJ^{[i+1]}$. Thus we get an ascending chain 
$(\cA_\cJ^{[i]})_{i\in \Z}$ of full subcategories of $\cA$. If $\cJ$ is 
stable under truncations and the t-structure is bounded, then $L\in \cD$ 
belongs to $\cJ$ if and only if $H^iL\in \cA_\cJ^{[i]}$ for all $i\in \Z$. 
Recall that a full subcategory of an Abelian category stable under 
subquotients and extensions is called a \emph{Serre subcategory}. 

\begin{defn}
We say that $\cJ$ is \emph{compatible} with the t-structure on $\cD$ if $\cJ$ 
is stable under truncations and $\cA_\cJ^{[i]}$ is a Serre subcategory of 
$\cA$ for all $i\in \Z$. 
\end{defn}

Conversely, given an ascending chain $\cI=(\cI^i)_{i\in \Z}$ of full 
subcategories of $\cA$, we let $\cD_\cI$ denote the full subcategory 
consisting of objects $L$ of $\cD$ such that $H^iL\in \cI^i$ for all $i\in 
\Z$. Then $\cD_\cI\subseteq \cD_\cI[1]$. If $(\cI^i)_{i\in \Z}$ is an 
ascending chain of Serre subcategories of $\cA$, then $\cD_\cI$ is stable 
under extensions and compatible with the t-structure on $\cD$. 

The following lemma is straightforward.

\begin{lemma}
Let $\cD$ be a triangulated category equipped with a bounded t-structure of 
heart~$\cA$. The constructions $\cI\mapsto \cD_\cI$ and $\cJ\mapsto 
(\cA_\cJ^{[i]})_{i\in \Z}$ define bijections, inverse to each other, between 
the class of ascending chains of Serre subcategories of $\cA$ and the class 
of full subcategories $\cJ\subseteq \cD$ satisfying $\cJ\subseteq \cJ[1]$, 
stable under extensions and compatible with the t-structure. 
\end{lemma}

\begin{theorem}\label{t.trunc}
Let $k$ be a field and let $X$ be a scheme of finite type over $k$. Let 
$\ell$ be a prime number invertible in $k$. For each subscheme $V\subseteq 
X$, suppose we are given a full subcategory $\cJ(V)\subseteq D^b_c(V,\Qlb)$ 
satisfying $\cJ(V)\subseteq \cJ(V)[1]$, stable under extensions and 
compatible with the usual t-structure. Assume that for every immersion 
$u\colon U\to V$ of subschemes of $X$, $u^*$ preserves $\cJ$. We put 
$\Perv^{[i]}_\cJ(V)=\Perv_c(V,\Qlb)\cap (\cJ(V)[i])$. Consider the following 
conditions. 
\begin{enumerate}[(A)]
\item $Rj_*$ preserves $\cJ$ for every immersion $j\colon V\to X$.
\item $j_{!*}$ preserves $\Perv_\cJ^{[i]}$ for every immersion $j\colon  
    V\to X$ and every $i\in \Z$. 
\item  $j_!$ preserves $\cJ$ for every immersion $j\colon V\to X$.
\item $\Perv_\cJ^{[i]}(V)=\Perv_\cJ^{[i]}(V)'$ for every subscheme $V$ of 
    $X$. Here $\Perv_\cJ^{[i]}(V)'\subseteq \Perv_c(V,\Qlb)$ denotes the 
    full subcategory consisting of perverse sheaves $\cP$ such that for 
    every irreducible subscheme $U$ of $V$, there exists a dense open 
    subscheme $W$ of $U$ such that $\cH^{-d} \cP|_W\in \cJ(W)[i-d]$, where 
    $d=\dim(U)$. 
\item $\cJ(V)\subseteq D^b_c(V,\Qlb)$ is compatible with the perverse 
    t-structure for every subscheme $V$ of $X$. 
\item $\Perv_\cJ^{[i]}(V)\subseteq \Perv_c(V,\Qlb)$ is a Serre subcategory 
    for every subscheme $V$ of $X$ and every $i\in \Z$. 
\end{enumerate}
We have
\[\text{(A)} \implies \text{(B)} \implies \text{(D)} \implies \text{(E)} \implies \text{(F)}.\]
Moreover, (B)$\iff$(C)+(F).
\end{theorem}

\begin{remark}
Condition (A) implies that $Ru_*$ preserves $\cJ$ for any immersion 
    $u\colon U\to V$ of subschemes of $X$. Indeed, $Ru_*\simeq j^*R(ju)_*$. 
    The same remark can be made for conditions (B) and (C). 
\end{remark}

\begin{lemma}\label{l.D}
$\Perv_\cJ^{[i]}(V)'\subseteq \Perv_c(V,\Qlb)$ is stable under quotients.
\end{lemma}

\begin{proof}
Let $\cP\in \Perv_\cJ^{[i]}(V)'$ and let
\[
0\to \cP'\to \cP\to \cP''\to 0
\]
be a short exact sequence of perverse sheaves. Let $U$ be an irreducible 
subscheme of $V$. By assumption, there exists a dense open subscheme $W$ of 
$U$ such that $\cH^{-d} \cP|_W\in \cJ(W)[i-d]$, where $d=\dim(U)$. Let $S$ be 
the support of $\cH^{-d+1}\cP'$, which has dimension $\le d-1$. Then the 
above short exact sequence induces an epimorphism $\cH^{-d}\cP|_{W\backslash 
S} \to \cH^{-d}\cP''|_{W\backslash S}$, which implies that
$\cH^{-d}\cP''|_{W\backslash S}\in \cJ(W\backslash S)[i-d]$. This finishes 
the proof that $\cP''\in \Perv_\cJ^{[i]}(V)'$. 
\end{proof}

\begin{lemma}\label{l.Zar}
Assume (C) holds. Then $\cJ$ satisfies Zariski descent. In other words, for 
every subscheme $V$ of $X$ with a Zariski open cover $(U_\alpha)$ of $V$ and 
every $L\in D^b_c(V,\Qlb)$ satisfying $L|_{U_\alpha}\in \cJ(U_\alpha)$ for 
all $\alpha$, we have $L\in \cJ(V)$.
\end{lemma}

\begin{proof}
By the assumption that $\cJ$ is compatible with the usual t-structure, we may 
assume that $L\in \Shv_c(V,\Qlb)$. We may assume further that the cover is 
finite. Let $u_\alpha\colon U_\alpha\to V$ be the open immersion. Then $L$ is 
a quotient of $\bigoplus_{\alpha} u_{\alpha!}u_\alpha^*L$, which is in 
$\cJ(V)$ by (C). 
\end{proof}

\begin{proof}[Proof of Theorem \ref{t.trunc}]
We will show that (B) is equivalent to the following weakened form:
\begin{enumerate}
\item[(wB)] For every immersion $j\colon  V\to X$ with $V$ \emph{affine}, 
    every $i\in \Z$, and every \emph{simple} perverse sheaf $\cP\in 
    \Perv_\cJ^{[i]}(V)$, we have $j_{!*}\cP\in \Perv_\cJ^{[i]}(X)$.
\end{enumerate}

(B)$\implies$(wB). Trivial.

(wB)$\implies$(D). Condition (D) is the analogue of \cite[Corollaire 
5.3.7]{BBD}. There it is proved as a corollary of the existence of the weight 
filtration \cite[Proposition 5.3.5]{BBD}, which has no analogue in our 
setting, but we can adapt the proof as follows. Clearly $\Perv_\cJ^{[i]}(V) 
\subseteq \Perv_\cJ^{[i]}(V)'$. For the other direction, let $\cP\in 
\Perv_\cJ^{[i]}(V)'$. We prove $\cP\in \Perv_\cJ^{[i]}(V)$ by induction on 
the length of $\cP$. The case $\cP=0$ is trivial. Assume that $\cP$ is 
nonzero. Then there exists a short exact sequence of perverse sheaves 
\[
0\to \cP'\to \cP\to \cP''\to 0
\]
with $\cP''$ simple. There exists an immersion $u\colon U\to V$ with $U_\red$ 
irreducible regular of dimension $d$ such that $\cP''=u_{!*}(\cL[d])$, where 
$\cL$ is a lisse sheaf. By Lemma \ref{l.D}, $\cP''\in \Perv_\cJ^{[i]}(V)'$. 
Thus, up to replacing $U$ by a dense open subscheme, we may assume that $U$ 
is affine and $\cL\simeq \cH^{-d}\cP''|_U\in \cJ(U)[i-d]$. Then $\cP''\in 
\cJ(V)[i]$ by (wB). For any $d'\in \Z$, the above short exact sequence 
induces an exact sequence 
\[\cH^{-d'-1}\cP''\to \cH^{-d'}\cP'\to 
\cH^{-d'}\cP\] with $\cH^{-d'-1}\cP''\in \cJ(V)[i-d'-1]\subseteq \cJ(V)[i-d']$. 
The assumption $\cP\in \Perv_\cJ^{[i]}(V)'$ then implies $\cP'\in 
\Perv_\cJ^{[i]}(V)'$. By induction hypothesis, it follows that $\cP'\in 
\cJ(V)[i]$. Therefore, $\cP\in \cJ(V)[i]$. 

(D)$\implies$(E). Clearly $\Perv_\cJ^{[i]}(V)$ is stable under extensions. By 
(D) and Lemma \ref{l.D}, $\Perv_\cJ^{[i]}(V)$ is also stable under quotients. 
The stability under subobjects follows formally: If $\cP\in 
\Perv_\cJ^{[i]}(V)$ and $\cP'$ is a perverse subsheaf of $\cP$, then 
$\cP/\cP'\in \cJ(V)[i]$ and the distinguished triangle $\cP'\to \cP \to 
\cP/\cP'\to \cP'[1]$ implies that $\cP'\in \cJ(V)[i]$. Thus 
$\Perv_\cJ^{[i]}(V)$ is a Serre subcategory. It remains to show that $\cJ(V)$ 
is stable under perverse truncations. The proof is identical to that of 
\cite[Th\'eor\`eme 5.4.1]{BBD}. Let $L\in \cJ(V)$. We proceed by descending 
induction on $i$ to show $\ptau^{\ge i}L\in \cJ(V)$, which implies $\ptau^{< 
i}L\in \cJ(V)$. For $i$ large enough, $\ptau^{\ge i} L=0$ and the assertion 
is trivial. Assume now that $\ptau^{\ge i+1}L$ and $\ptau^{<i+1}L=\ptau^{\le 
i}L$ are in $\cJ(V)$. Let $U$ be an irreducible subscheme of $V$ of dimension 
$d$. Let $S$ be the support of $\cH^{i-d+1}(\ptau^{<i} L)$, which has 
dimension $\le d-2$. The distinguished triangle 
\[\ptau^{<i} L \to \ptau^{\le i}L\to (\pH^i L)[-i]\to 
(\ptau^{<i} L)[1]
\] 
induces an epimorphism $\cH^{i-d}(\ptau^{\le i}L)|_{U\backslash S}\to 
\cH^{-d}(\pH^i L)|_{U\backslash S}$. Since $\cH^{i-d}(\ptau^{\le i}L)\in 
\cJ(U)[i-d]$, this implies $\cH^{-d}(\pH^i L)|_{U\backslash S}\in 
\cJ(U\backslash S)[i-d]$. Thus $\pH^i L\in 
\Perv_\cJ^{[i]}(V)'=\Perv_\cJ^{[i]}(V)$ by (D). It then follows from the 
distinguished triangle $(\pH^i L)[-i]\to \ptau^{\ge i}L \to \ptau^{\ge{i+1}} 
L\to (\pH^i L)[-i+1]$ that $\ptau^{\ge i} L\in \cJ(V)$. 

(E)$\implies$(F). Trivial.

(wB)$\implies$(C). It suffices to show that $j_!\cF\in \cJ(X)[i]$ for $\cF\in 
\Shv_c(V,\Qlb)\cap (\cJ(V)[i])$. Note that $\cF$ is a successive extension of 
sheaves of the form $u_!\cL$, where $u\colon U\to V$ is an immersion with 
$U_\red$ irreducible regular affine and $\cL=u^*\cF$ is a simple lisse sheaf 
on $U$. Thus we may assume that $V_\red$ is irreducible regular affine and 
$\cF$ is a simple lisse sheaf. In this case $j_!\cF$ is a subsheaf of $\cH^0 
(j_{!*}(\cF[d])[-d])$, which is in $\cJ(X)[i]$ by (wB). Here $d=\dim(V)$. It 
follows that $j_!\cF$ is in $\cJ(X)[i]$. 

(C)+(F)$\implies$(D). For any affine immersion $u\colon U\to V$ of subschemes 
of $X$ and any $\cP\in \Perv_\cJ^{[i]}(U)$, we have $u_!\cP\in 
\Perv_\cJ^{[i]}(V)$ by (C), which implies that the quotient $u_{!*}\cP$ is in 
$\Perv_\cJ^{[i]}(V)$ by (F). Since (wB)$\implies$(D), this implies  
$\Perv_\cJ^{[i]}(V)=\Perv_\cJ^{[i]}(V)'$ for every separated subscheme $V$ of 
$X$ (or more generally for $V$ of affine diagonal). To prove (D) in general, 
choose a Zariski open cover $(U_\alpha)$ of $V$ with $U_\alpha$ separated for 
all $\alpha$. It suffices to prove the inclusion 
$\Perv_\cJ^{[i]}(V)'\subseteq \Perv_\cJ^{[i]}(V)$. Let $\cP\in 
\Perv_\cJ^{[i]}(V)'$. Then $\cP|_{U_\alpha}\in 
\Perv_\cJ^{[i]}(U_\alpha)'=\Perv_\cJ^{[i]}(U_\alpha)$ for every $\alpha$, 
which implies $\cP\in \Perv_\cJ^{[i]}(V)$ by Lemma \ref{l.Zar}. 

(C)+(E)$\implies$(B). This is clear, because $j_{!*}\cP$ is a quotient of 
$\pH^0j_!\cP$. 

(A)$\implies$(B). Since we have already proven 
(wB)$\implies$(C)+(F)$\implies$(B), it suffices to show (wB). Let $j\colon 
V\to X$ be an immersion and let $\cP$ be a simple perverse sheaf in 
$\Perv_\cJ^{[i]}(V)$. Up to replacing $V$ by a subscheme, we may assume that 
$V$ is irreducible and $\cP\in \Shv_c(V,\Qlb)[d]$, where $d=\dim(V)$. Let 
$V_0$ be the closure of $V$ in $X$. Then $j$ can be decomposed into a 
sequence of immersions 
\[V=V_d\xrightarrow{j_{d}} V_{d-1}\to \dots\to V_1 \xrightarrow{j_1} V_0\xrightarrow{j_0}X\]
such that $(V_{n-1}-V_{n})_\red$ is regular of dimension $n-1$ and 
$(R(j_{n}j_{n+1}\dotsm j_{d})_{*}\cP)|_{V_{n-1}-V_n}$ has lisse cohomology 
sheaves for all $1\le n\le d$. In this case, we have
\[j_{!*}\cP\simeq Rj_{0*}\tau^{\le -1}Rj_{1*}\dotsm \tau^{\le -d}Rj_{d*} \cP,\]
which is in $\cJ(X)[i]$ by (A).
\end{proof}

\begin{example}\label{e.wt}
Let $k=\F_q$ be a finite field. For any scheme $X$ of finite type over~$k$, 
let $\cJ(X)\subseteq D^b_c(X,\Qlb)$ be the full subcategory consisting of 
mixed complexes of weight $\le 0$. Then conditions (B), (C), (D), (E), (F) of 
Theorem \ref{t.trunc} hold. In fact, (C) is trivial and (F) is 
\cite[Proposition 5.3.1]{BBD}. Condition (A) does not hold in this case.
\end{example}

Our main application of Theorem \ref{t.trunc} is to integral complexes. Let 
us recall the definition and basic properties. For completeness, we include 
both the case over a finite field and the case over a local field. For the 
proofs of our  theorems about rigid spaces, only the finite field case is 
needed. 

\begin{defn}\label{d.int}
Let $X$ be a scheme of finite type over $\F_q$ (resp.\ a local field $K$ of 
residue field $\F_q$) and let $\ell\nmid q$ be a prime number. 
\begin{enumerate}
\item We say that $\cF\in \Shv_c(X,\Qlb)$ is \emph{integral} if for every 
    geometric point $\bar x$ above a closed point $x$ of $X$, the 
    eigenvalues of the geometric Frobenius $\Fr_x\in G_x$ (resp.\ of an 
    element $g_x\in G_x$ projecting to the geometric Frobenius) acting on 
    $\cF_{\bar x}$ are algebraic integers. 
\item We say that $\cF\in \Shv_c(X,\Qlb)$ is \emph{inverse integral} if for 
    every geometric point $\bar x$ above a closed point $x$ of $X$, the 
    eigenvalues of the geometric Frobenius $\Fr_x\in G_x$ (resp.\ of an 
    element $g_x\in G_x$ projecting to the geometric Frobenius) acting on 
    $\cF_{\bar x}$ are inverses of algebraic integers. 
\item We say that a complex $L\in D^b_c(X,\Qlb)$ is \emph{integral} if 
    $\cH^iL$ is integral for all $i\in \Z$. 
\item We say that a complex $L\in D^b_c(X,\Qlb)$ is \emph{$\Id$-inverse 
    integral} if $\cH^iL(i)$ is inverse integral for all $i\in \Z$. 
\end{enumerate} 
\end{defn}

As a consequence of Grothendieck's monodromy theorem, up to multiplication by 
roots of unity, eigenvalues of a lift of geometric Frobenius do not depend on 
the choice of the lift \cite[Lemme 1.7.4]{WeilII}. Thus the above definition 
in the local field case does not depend on the choice of $g_x$.

\begin{remark}
In \cite{Zint}, we consider more generally integrality and inverse 
integrality over the field of fractions $K_0$ of an excellent Henselian 
discrete valuation ring of residue field $\F_q$. The completion 
$\widehat{K_0}$ of $K_0$ is a local field. By a valuative criterion for 
integrality \cite[Corollary 3.10]{comp}, for a scheme $X$ of finite type over 
$K_0$, $\cF\in \Shv_c(X,\Qlb)$ is integral if and only if $\pi^*\cF$ is 
integral, where $\pi\colon X_{\widehat{K_0}}\to X$ is the projection. In 
fact, the valuative criterion also holds for inverse integrality by a similar 
proof. In particular, $\cF\in \Shv_c(X,\Qlb)$ is inverse integral if and only 
if $\pi^*\cF$ is inverse integral. Therefore, there is no loss of generality 
to work over a (complete) local field. 
\end{remark}

\begin{theorem}\label{t.push}
For any morphism $f\colon X\to Y$ of schemes of finite type over $\F_q$ or a 
local field $K$, integral complexes are stable under $Rf_!$ and $Rf_*$. The 
same holds for $\Id$-inverse integral complexes. 
\end{theorem}

The case of $Rf_!$ for $f$ separated over $\F_q$ is a theorem of Deligne 
\cite[XXI, Th\'eor\`emes 5.2.2, 5.4]{SGA7II} and the integral case of $Rf_!$ 
over $K$ is a theorem of Deligne and Esnault \cite[Appendix, Theorem 
0.2]{Esnault}. The other cases are proved in \cite[Proposition 6.4]{Zint}. 

\begin{example}\label{e.int}
Let $k$ be a finite or a local field.
\begin{enumerate}
\item For any scheme $X$ of finite type over~$k$, let $\cJ(X)\subseteq 
    D^b_c(X,\Qlb)$ be the full subcategory consisting of integral 
    complexes. Then condition (A) and hence all conditions of Theorem 
    \ref{t.trunc} hold for $\cJ$. 
    
\item For any scheme $X$ of finite type over~$k$, let $\cJ'(X)\subseteq 
    D^b_c(X,\Qlb)$ be the full subcategory consisting of $\Id$-inverse 
    integral complexes. Then condition (A) and hence all conditions of 
    Theorem \ref{t.trunc} hold for $\cJ'$. 
\end{enumerate}

In particular, for any scheme $X$ of finite type over $k$ of dimension $d$, 
$\IC_{X,\Qlb}[-d]$ is integral and $\Id$-inverse integral by (B). 
\end{example}

Let us spell out condition (E), of which the finite field case will be used 
in the proof of Theorem \ref{t.psi}. 

\begin{cor}\label{c.trunc}
Let $X$ be a scheme of finite type over~$\F_q$ or a local field $K$. 
\begin{enumerate}
\item Integral complexes are stable under perverse truncations and integral 
    perverse sheaves are stable under subquotients. 
\item $\Id$-inverse integral complexes are stable under perverse 
    truncations and $\Id$-inverse integral perverse sheaves are stable 
    under subquotients. 
\end{enumerate}
\end{cor}

Now we discuss algebraic analogues of Theorem \ref{t.main}.

\begin{remark}\label{r.sch}
Let $k=\F_q$ (resp.\ $k=K$ be a local field of residue field $\F_q$). Let $X$ 
be a scheme of dimension $d$ of finite type over $k$. Let $\ell\neq 
p=\chara(\F_q)$ be a prime number.  Let $g=\Fr_q$ (resp.\ $g\in G_K$ be an 
element projecting to $\Fr_{q}$). For all $i\in \Z$, the eigenvalues of $g$ 
acting on $H^i_{(c)}(X_{\bar k},\Q_\ell)$ and $\IH^{i-d}_{(c)}(X_{\bar 
k},\Q_\ell)$ are algebraic integers that are units away from $p$ and their 
$q$-slopes are contained in $[0,i]\cap [i-d,d]$. 

The case of $H^i_{(c)}(X_{\bar k},\Q_\ell)$ is proved in \cite[XXI, 
Th\'eor\`emes 5.2.2, 5.4]{SGA7II},  \cite[Appendix, Theorem 0.2]{Esnault}, 
\cite[Proposition 6.4]{Zint}. We point out two proofs of the case of 
$\IH^{i-d}_{(c)}(X_{\bar k},\Q_\ell)$, both relying on Theorem \ref{t.trunc} 
at least in the local field case. We may assume that $X$ has pure dimension 
$d$.

Let $a\colon X\to \Spec(k)$ be the structural morphism. By Example 
\ref{e.int} and Theorem \ref{t.push}, $Ra_*\IC_{X,\Qlb}[-d]$ and 
$Ra_!\IC_{X,\Qlb}[-d]$ are integral and $\id$-inverse integral. This means 
that the eigenvalues of $g$ acting on $\IH^{i-d}_{(c)}(X_{\bar k},\Q_\ell)$ 
are algebraic integers that are units away from $p$ and their $q$-slopes are 
contained in $[0,i]$. The remaining bounds for $q$-slopes follow by duality: 
We have a $G_k$-equivariant perfect pairing $\IH^{i-d}(X_{\bar 
k},\Q_\ell)\times \IH^{d-i}_{c}(X_{\bar k},\Q_\ell)\to \Q_\ell(-d)$.

Alternatively, we can choose a projective alteration $f\colon Y\to X$ with 
$Y$ smooth \cite{dJ}. In the finite field case, by the decomposition and hard 
Lefschetz theorems of Beilinson--Bernstein--Deligne--Gabber \cite[Corollaire 
5.3.11, Th\'eor\`eme 5.4.10]{BBD}, $\IC_{X,\Q_\ell}$ is a direct summand of 
$Rf_*\Q_\ell[d]$, which implies that $\IH^{i-d}_{(c)}(X_{\bar k},\Q_\ell)$ is 
a $G_k$-equivariant direct summand of $H^i_{(c)}(Y_{\bar k},\Q_\ell)$. The 
same holds in the local field case by the formalism of horizontal complexes 
introduced by A.~Huber \cite{HuberA} and extended by Morel \cite{Morel}. See 
\cite[Section 5.1]{SZ} for the analogue of \cite[Corollaire 5.3.11]{BBD} and 
below for the hard Lefschetz theorem. 
\end{remark}

\begin{remark}
Let $k$ be a finitely generated field. Then the analogue of Theorem 
\ref{t.trunc} holds for the derived category $D^b_h(X,\Qlb)$ of horizontal 
complexes. This implies the hard Lefschetz theorem for horizontal perverse 
sheaves (Corollary \ref{c.HL}) as follows. 

Let $\cJ(X)\subseteq D^b_h(X,\Qlb)$ be the full subcategory consisting of 
mixed complexes of weight $\le 0$ as in Example \ref{e.wt}.  Then as before 
conditions (B), (C), (D), (E), (F) of Theorem \ref{t.trunc} hold. In fact, 
(C) is trivial and (F) follows from \cite[Proposition 5.3.1]{BBD} (see 
\cite[Proposition 3.4]{HuberA} and \cite[Section 2.6]{Morel}). In particular, 
even though the weight filtration does not exist for horizontal perverse 
sheaves in general, conditions (D) and (E), which are analogues of 
\cite[Corollaire 5.3.7, Th\'eor\`eme 5.4.1]{BBD}, hold for horizontal 
complexes. One can then repeat the proof of the hard Lefschetz theorem 
\cite[Th\'eor\`eme 5.4.10]{BBD}, which relies on condition (E). 
\end{remark}

\begin{cor}[Hard Lefschetz theorem for horizontal perverse 
sheaves]\label{c.HL}
Let $f\colon X\to Y$ be a projective morphism of schemes of finite type over 
a finitely generated field $k$ and let $\cL$ be an $f$-ample line bundle on 
$X$. Let $\ell$ be a prime number invertible in $k$ and let $\cP$ be a pure 
horizontal $\Qlb$-perverse sheaf on $X$. Then, for all $i\ge 0$, cupping with 
$c_1(\cL)^i$ induces an isomorphism 
\[\pH^{-i}Rf_*\cP\to \pH^{i}Rf_*\cP(i).\]
\end{cor}

This result was tacitly used in the proof of \cite[Theorem 5.1.5]{SZ}. 

\section{Integrality and nearby cycles}\label{s.3}

Let $K$ be a local field of residue field $\F_q$. Let $\ell\nmid q$ be a 
prime number. We recall the following extension of the notion of integral 
complexes to the fiber product topos on which nearby cycles live. 

\begin{defn}
Let $Y$ be a scheme of finite type over $\F_q$. We say that a complex $L\in 
D^b_c(Y\times_s \eta,\Qlb)$ is \emph{integral} (resp.\ \emph{$\Id$-inverse 
integral}) if $\sigma^* L$ is integral (resp.\ $\Id$-inverse integral). Here 
$s=\Spec(\F_q)$, $\eta=\Spec(K)$, and $\sigma\colon G_{\F_q}\to G_{K}$ is a 
continuous section of the reduction homomorphism $r\colon G_K\to G_{\F_q}$. 
\end{defn}

The above definition does not depend on the choice of the section $\sigma$. 
Using this definition, we can restate Theorem \ref{t.psi} succinctly as part 
(b) of the following. 

\begin{theorem}\label{t.psi2}
Let $\cX$ be an admissible formal $\cO_K$-scheme with generic fiber 
$\cX_\eta$ of dimension $d$.
\begin{enumerate}
\item If $\cX_\eta$ is smooth, then $R\Psi_{\cX}\Qlb$ is integral and 
    $\Id$-inverse integral. 
    
\item Assume $\chara(K)=0$. Then $R\Psi_{\cX}\Qlb$ and 
    $R\Psi_{\cX}\IC_{\cX_\eta,\Qlb}[-d]$ are integral and $\Id$-inverse 
    integral. 
\end{enumerate}
\end{theorem}

In the proof of Theorem \ref{t.psi2}, we will use the following theorem on 
algebraic nearby cycles \cite[Th\'eor\`eme 5.3]{Zint}. 

\begin{theorem}\label{t.psialg}
Let $X$ be a scheme of finite type over $\cO_K$. Then 
\[R\Psi^\alg_{X}\colon 
D^b_c(X_\eta,\Qlb)\to D^b_c(X_s\times_s \eta,\Qlb)
\] 
sends integral complexes to integral complexes and $\Id$-inverse integral 
complexes to $\Id$-inverse integral complexes. In particular, 
$R\Psi^\alg_{X}\Qlb$  is integral and $\Id$-inverse integral. 
\end{theorem}

\begin{remark}
Let $X$ be a scheme of finite type over $\cO_K$ with generic fiber $X_\eta$ 
of dimension $d$. Then $R\Psi^\alg_{X}\IC_{X_\eta,\Qlb}[-d]$ is also integral 
and $\Id$-inverse integral, by Example \ref{e.int} and Theorem 
\ref{t.psialg}. 
\end{remark}

\begin{remark}
We correct two inaccuracies in \cite[Section 5]{Zint} about base change of 
semistable pairs. In the proof of \cite[Lemme 5.6]{Zint}, on line 5 of page 
486, the finite morphism $h\colon X'=\coprod_\zeta X_\zeta \to X_{S'}$ is an 
isomorphism over the generic point of $S'$, but not an isomorphism in 
general. The pair $(X',(g_{S'}h)^{-1}(Z_{S'})_\red)$ is a strictly semistable 
pair over $S'$. 

In \cite[Lemme 5.10(ii)]{Zint}, the condition on $r_m\colon X'_m\to X_{S'}$ 
should be the following: There exists a decomposition of $f$ into morphisms 
of traits $S'\to S_m\to S$ such that $r_m$ is the base change of a proper 
morphism $X_m\to X_{S_m}$ satisfying \cite[5.9.1]{Zint}. 
\end{remark}

\begin{proof}[Proof of Theorem \ref{t.psi2}]
We closely follow the strategy in \cite[Lemmas 3.1.9, 3.2.4]{HZ} to reduce to 
the algebraizable case. 

\emph{Case (a).} Since the problem is local, we may assume that $\cX$ is 
affine. Then, by \cite[Theorem 3.1.3]{Temkin17}, $\cX$ is algebraizable. In 
other words, $\cX$ is the $\varpi$-adic completion of a flat affine 
$\cO_K$-scheme $X$ of finite type, where $\varpi$ is a uniformizer of 
$\cO_K$. By \cite[Theorem A.4.4 and Remark B.1]{HZ}, $R\Psi_{\cX}\Qlb\simeq 
R\Psi^\alg_{X}\Qlb$, which is integral and $\Id$-inverse integral by Theorem 
\ref{t.psialg}. 

\emph{Case (b) for $R\Psi_{\cX}\Qlb$.} By a theorem of Temkin \cite[Theorem 
5.2.2]{Temkin12}, for any admissible formal $\cO_K$-scheme $\cZ$, 
$(\cZ_\eta)_\red$ admits a resolution of singularities, which extends to a 
rig-surjective morphism of admissible formal $\cO_K$-schemes $\cZ'\to \cZ$ 
with $\cZ'_\eta$ smooth. Using this and a standard procedure 
\cite[6.2.5]{HodgeIII}, one constructs a rig-surjective hypercover 
$f_\bullet\colon \cY_\bullet\to \cX$ with $\cY_{i,\eta}$ smooth for all $i\ge 
0$. By cohomological descent, $\Qlb\simeq Rf_{\bullet\eta*}\Qlb$ \cite[Lemma 
3.1.3]{HZ}. Thus 
\[R\Psi_{\cX}\Qlb\simeq R\Psi_{\cX}Rf_{\bullet\eta*}\Qlb\simeq R(f_{\bullet s}\times_s \eta)_*R\Psi_{\cY_\bullet}\Qlb,\]
which induces a spectral sequence 
\[E_{1}^{i,j}=R^j(f_{i, s}\times_s 
\eta)_*R\Psi_{\cY_i}\Qlb\Rightarrow R^{i+j}\Psi_{\cX}\Qlb.
\] 
By (a) and Theorem \ref{t.push}, $R(f_{i, s}\times_s 
\eta)_*R\Psi_{\cY_i}\Qlb$ is integral and $\Id$-inverse integral. It follows 
that $R\Psi_{\cX}\Qlb$ is integral and $\Id$-inverse integral. 

\emph{Case (b) for $R\Psi_{\cX}\IC_{\cX_\eta,\Qlb}[-d]$.} As in the proof of 
\cite[Lemma 3.2.3]{HZ}, up to replacing $\cX$ by a connected component of its 
normalization, we may assume that $\cX_\eta$ is reduced of pure dimension 
$d$. By \cite[Theorem 5.2.2]{Temkin12}, $\cX_\eta$ admits a resolution of 
singularities, which extends to a morphism of admissible formal 
$\cO_K$-schemes $f\colon \cY\to \cX$ with $\cY_\eta$ smooth. Then 
$\IC_{\cX_\eta,\Qlb}$ is a subquotient of $\pH^0 Rf_{\eta*}(\Qlb[d])$. By the 
perverse exactness of $R\Psi_\cX$ \cite[Theorem 4.11]{BH}, 
$R\Psi_{\cX}\IC_{\cX_\eta,\Qlb}$ is a subquotient of 
\[\pH^0 (R\Psi_{\cX}Rf_{\eta*}\Qlb[d])\simeq \pH^0(R(f_{s}\times_s \eta)_*R\Psi_{\cY}\Qlb[d]).\] 
By (a) and Theorem \ref{t.push}, $R(f_{s}\times_s \eta)_*R\Psi_{\cY}\Qlb$ is 
integral and $\Id$-inverse integral. By Corollary \ref{c.trunc}, it follows 
that $R\Psi_{\cX}\IC_{\cX_\eta,\Qlb}[-d]$ is integral and $\Id$-inverse 
integral. 
\end{proof}

The following continuity property for compactly supported cohomology will be 
used in the proof of Theorem \ref{t.main}. 

\begin{lemma}\label{l.cont}
Let $X$ be a qcqs rigid space over a $p$-adic local field $K$ and let $C$ be 
a completed algebraic closure of $K$. Let $V\subseteq X$ be a Zariski open 
subset. Then, for every $i\in \Z$, we have a canonical $G_K$-equivariant 
isomorphism 
\[\varinjlim_{U} H^i_c(U_C,\Q_\ell)\simeq H^i_c(V_C,\Q_\ell),\]
where $U$ runs through quasi-compact open subsets of $V$.
\end{lemma}

\begin{proof}
In the separated case, this is a result of R.~Huber \cite[Proposition 
2.1(iv)]{HuberComp}. In the general case, choose a finite cover of $X$ by 
quasi-compact separated open subsets $X_1,\dots, X_n$. Let 
$X^{(i)}=\coprod_{\alpha_0<\dots <\alpha_i} X_{\alpha_0}\cap \dotsm \cap 
X_{\alpha_i}$ for $0\le i< n$ and $X^{(i)}=\emptyset$ otherwise. Using the 
spectral sequences 
\[E_1^{ij}=H^j_c(X^{(-i)}\times_X W_C,\Q_\ell)\Rightarrow H^{i+j}_c(W_C,\Q_\ell),\]
for $W=U$ and $W=V$, one reduces the problem to the known case where $X$ is 
separated. 
\end{proof}

\begin{proof}[Proof of Theorem \ref{t.main}]
Let $\cX$ be an admissible formal model of $X$ with structure morphism 
$a\colon \cX\to \Spf(\cO_K)$. Then
\begin{equation}\label{e.maintriv} 
Ra_{\eta*}\simeq R(a_s\times_s\eta)_*R\Psi_{\cX}
\end{equation}
and 
\[Ra_{\eta!}L\simeq 
R(a_s\times_s\eta)_!R\Psi_{\cX}L
\] 
for $L=\Qlb$ and $L=\IC_{X,\Qlb}[-d]$ by \cite[Theorem A.3.9]{HZ}. Indeed, we 
have 
\[D_\eta Ra_{\eta*} D_X\simeq D_\eta R(a_s\times_s\eta)_*R\Psi_{\cX}D_X\simeq D_\eta R(a_s\times_s\eta)_* D_{\cX_s\times_s \eta}R\Psi_{\cX}\]
by \eqref{e.maintriv} and \cite[Theorem 4.4(1)]{GW} (or \cite[Lemma 
A.3.8(4)]{HZ}). By Theorem \ref{t.psi2}, $R\Psi_\cX L$ is integral and 
$\Id$-inverse integral. It then follows from Theorem \ref{t.push} that 
\[Ra_{\eta*}\Qlb,\quad 
Ra_{\eta!}\Qlb,\quad Ra_{\eta*}\IC_{X,\Qlb}[-d],\quad Ra_{\eta!}\IC_{X,\Qlb}[-d]
\]
are all integral and $\Id$-inverse integral. This means that the eigenvalues 
of $g$ acting on $H^i_{(c)}(X_C,\Q_\ell)$ and $\IH^{i-d}_{(c)}(X_C,\Q_\ell)$ 
are algebraic integers that are units away from $p$ and their $q$-slopes are 
contained in $[0,i]$. 

The remaining bounds for $q$-slopes for $\IH^*_{(c)}(X_C,\Q_\ell)$ follow 
from duality. Indeed, up to replacing $X$ by its normalization, we may assume 
that $X$ is of pure dimension $d$. In this case we have 
$D_{X}\IC_{X,\Q_\ell}\simeq \IC_{X,\Q_\ell}(d)$ by \cite[proof of Theorem 
4.13(3)]{BH}, which induces a $G_K$-equivariant perfect pairing 
$\IH^{i-d}(X_C,\Q_\ell)\times \IH^{d-i}_{c}(X_C,\Q_\ell)\to \Q_\ell(-d)$. 
Thus the $q$-slopes of $\IH^{i-d}_{(c)}(X_C,\Q_\ell)$ are contained in 
$[i-d,d]$. In particular, in the case where $X$ is smooth, the $q$-slopes of 
$H^{i}_{(c)}(X_C,\Q_\ell)$  are contained in $[i-d,d]$. 

For the remaining bounds for $q$-slopes of $H^{i}_{c}(X_C,\Q_\ell)$ for 
general $X$, we reduce to the smooth case as follows, using the same argument 
as in the proof of \cite[Theorem 5.3]{Mieda}. We proceed by induction on the 
dimension $d$ of $X$. The case $X$ empty is trivial. Let $X$ be nonempty. We 
may assume that $X$ is reduced. Let $Y$ be the singular locus of $X$ and let 
$V=X\backslash Y$. Consider the exact sequence 
\[H^i_c(V_C,\Q_\ell)\to H^i_c(X_C,\Q_\ell)\to H^i_c(Y_C,\Q_\ell).\]
By induction hypothesis, the $q$-slopes of $H^{i}_{c}(Y_C,\Q_\ell)$  are 
contained in $[i-d_Y,d_Y]\subseteq [i-d,d]$, where $d_Y=\dim(Y)$. By Lemma 
\ref{l.cont} and the smooth case, the $q$-slopes of $H^i_c(V_C,\Q_\ell)$ are 
contained in $[i-d,d]$. Therefore, the $q$-slopes of $H^i_c(X_C,\Q_\ell)$ are 
contained in $[i-d,d]$. 
\end{proof}

In the smooth case, the first two paragraphs of the above proof hold over a 
local field of any characteristic, which implies the following. 

\begin{prop}\label{p.pos}
Let $X$ be a smooth qcqs rigid space over a local field $K$ of residue field 
$\F_q$ and let $p=\chara(\F_q)$. Let $\ell\neq p$ be a prime number. Let 
$g\in G_K$ be an element projecting to $\Fr_{q}$. For all $i\in \Z$, the 
eigenvalues of $g$ acting on $H^i(X_C,\Q_\ell)$ and $H^i_c(X_C,\Q_\ell)$ are 
algebraic integers that are units away from $p$ and their $q$-slopes are 
contained in $[0,i]\cap [i-d,d]$. 
\end{prop}

\begin{lemma}\label{l.wt}
Let $\alpha$ be a $q$-Weil number of weight $w$ that is a unit away from $p$, 
where $q=p^r$. Assume that the $q$-slopes of $\alpha$ are contained in 
$[a,b]$. Then $w\in [2a,2b]$. 
\end{lemma}

\begin{proof}
This follows immediately from the product formula.
\end{proof}

\begin{proof}[Proof of Corollary \ref{c.wt}]
By \cite[Theorem 1.2.5]{HZ}, the eigenvalues are Weil numbers. The weight 
bounds then follow from Theorem \ref{t.main} by Lemma \ref{l.wt}. 
\end{proof}

\begin{lemma}\label{l.qu}
Let $(V,\rho)$ be a nonzero continuous $\ell$-adic representation of $G_K$ 
and $n\ge 0$ the greatest integer such that $\gr_n^M V\neq 0$, where $M$ is 
the monodromy filtration. 
\begin{enumerate}
\item For any $g\in I_K$ such that $\rho(g)$ is unipotent, we have 
    $(\rho(g)-1)^{n+1}=0$.
\item If $g\in G_K$ projects to $\Fr_q\in G_{\F_q}$, then there exist 
    eigenvalues $\alpha$ and $\beta$ of $\rho(g)$  such that 
    $\beta=q^n\alpha$. 
\end{enumerate}
\end{lemma}

We call $n$ the \emph{quasi-unipotency index} of $V$. 

\begin{proof}
(a) Let $N\colon V(1)\to V$ be the monodromy operator. Then $N^{n+1}=0$. 
There exists an open subgroup $I_1\subseteq I_K$ such that $\rho(h)=\exp(N 
t_\ell(h))$ for all $h\in I_1$. Here $t_\ell\colon I_K\to \Z_\ell(1)$ is the 
projection. Take an integer $m>0$ such that $g^m\in I_1$. Then 
$(\rho(g)^m-1)^{n+1}=0$. It follows that $(\rho(g)-1)^{n+1}=0$. 

(b) This is obvious, because $\gr_{-n}^M V\simeq \gr_{n}^M V(n)$.
\end{proof}

\begin{proof}[Proof of Corollary \ref{c.qu}]
By \cite[Theorem 3.3.1]{HZ}, there exists an open subgroup $I_1\subseteq 
I_K$, independent of $\ell$, such that for all $g\in I_1$, $(g-1)^N$ acts 
trivially on the cohomology groups, for some integer $N$. By Lemma 
\ref{l.qu}, we may replace $N$ by the quasi-unipotency index of the 
cohomology group and the bounds for $q$-slopes of Theorem \ref{t.main} imply 
the desired bounds for quasi-unipotency indices. 
\end{proof}

\section{Monodromy-pure perverse sheaves}\label{s.4}
Let $K$ be a local field of residue field $\F_q$ and let $\ell\nmid q$ be a 
prime number. In this section, we construct, for projective curves $Y$ over 
$s=\Spec(\F_q)$, monodromy-pure $\Q_\ell$-perverse sheaves on $Y\times_s 
\eta$ with non monodromy-pure cohomology. Here as usual $\eta=\Spec(K)$. The 
idea is to pick a monodromy operator that degenerates on the $E_2$-page of 
the weight spectral sequence. 

We start with a discussion of Galois descent for perverse sheaves. Let $Y$ be 
a scheme of finite type over $s$ and let $\bar s=\Spec(\overline{\F_q})$. We 
let $\Perv_c(Y_{\bar s},G_s,\Q_\ell)$ denote the category of 
$\Q_\ell$-perverse sheaves on $Y_{\bar s}$ equipped with an action of 
$G_s=G_{\F_q}$, compatible with the action of $G_s$ on $Y_{\bar s}$. The 
functor 
\[\Perv_c(Y,\Q_\ell)\to \Perv_c(Y_{\bar s},G_s,\Q_\ell)\]
sending $\cG$ to $(\cG|_{Y_{\bar s}},\rho)$, where $\rho$ is the obvious 
action, is fully faithful. We say that $(\cF,\rho)\in \Perv_c(Y_{\bar 
s},G_s,\Q_\ell)$ \emph{descends} to $Y$ if it is in the essential image of 
the above functor. 

Similarly, the functor 
\begin{equation}\label{e.descent}
\Perv_c(Y\times_s \eta,\Q_\ell)\to 
\Perv_c(Y_{\bar s},G_\eta,\Q_\ell)
\end{equation}
sending $\cP$ to $(\pi_Y^*\cP,\rho)$, where $\rho$ is the obvious action, is 
fully faithful. See for example \cite[Corollary B.2.6]{HZ}. Here $\pi_Y\colon 
Y_{\bar s}\to Y\times_s \eta $ and $\Perv_c(Y_{\bar s},G_\eta,\Q_\ell)$ 
denotes the category of $\Q_\ell$-perverse sheaves on $Y_{\bar s}$ equipped 
with an action of $G_\eta=G_{K}$, compatible with the action of $G_\eta$ on 
$Y_{\bar s}$ via the reduction homomorphism $r\colon G_\eta\to G_s$. For an 
object $(\cF,\rho)\in \Perv_c(Y_{\bar s},G_\eta,\Q_\ell)$, we say that the 
action $\rho|_{I_\eta}$ of the inertia group $I_\eta=I_K$ is 
\emph{continuous} if the map $\rho|_{I_\eta}\colon I_\eta\to \End(\cF)$ is 
continuous, where the finite-dimensional $\Q_\ell$-vector space $\End(\cF)$ 
is equipped with the usual topology. 

\begin{lemma}\label{l.descent}
Let $(\cF,\rho)\in \Perv_c(Y_{\bar s},G_\eta,\Q_\ell)$. The following 
conditions are equivalent. 
\begin{enumerate}
\item $(\cF,\rho)$ belongs to the essential image of \eqref{e.descent}. 
\item $\rho|_{I_\eta}$ is continuous and for every continuous section 
    $\sigma\colon G_s\to G_\eta$ of $r$, $(\cF,\sigma^*\rho)\in 
    \Perv_c(Y_{\bar s},G_s,\Q_\ell)$ descends to $Y$.
\item $\rho|_{I_\eta}$ is continuous and for some continuous section 
    $\sigma\colon G_s\to G_\eta$ of $r$, $(\cF,\sigma^*\rho)\in 
    \Perv_c(Y_{\bar s},G_s,\Q_\ell)$ descends to $Y$.
\end{enumerate}
Here $\sigma^*\rho$ denotes the restriction of $\rho$ to $G_s$ via $\sigma$. 
\end{lemma}

\begin{proof}
(a)$\implies$(b). Clearly $(\pi_Y^*\cP,\sigma^*\rho)$ descends to the 
perverse sheaf $\sigma^*\cP$ on $Y$. Let $a\colon Y\to s$ be the structure 
morphism. The sheaf $R^0(a\times_s \eta)_*R\cHom(r^*\sigma^*\cP,\cP)$ on 
$\eta$ corresponds to a continuous action $\alpha\colon G_\eta\times E\to E$, 
where $E=\Hom(\pi_Y^*\cP,\pi_Y^*\cP)$. The inertia $I_\eta$ acts trivially on 
the first $\pi_Y^*\cP$ and $I_\eta$ acts on the second $\pi_Y^*\cP$ via 
$\rho$. The map $\rho|_{I_\eta}$ can be identified with 
$\alpha|_{I_\eta\times \{\id_{\pi_Y^*\cP}\}}$, which is continuous. 

(b)$\implies$(c). Trivial.

(c)$\implies$(a). $(\cF,\sigma^*\rho)\in \Perv_c(Y_{\bar s},G_s,\Q_\ell)$ 
descends to a $\Q_\ell$-perverse sheaf $\cG$ on $Y$. Choose a torsion-free 
integral model of $\cG$. This corresponds to a torsion-free integral model of 
$\cF$, namely a torsion-free $\Z_\ell$-perverse sheaf $\cF_0$ on $Y$ 
satisfying $\cF_0\otimes_{\Z_\ell} \Q_\ell \simeq \cF$, stable under 
$\sigma(G_s)$. Since $\rho|_{I_\eta}$ is continuous, $\cF_0$ is also stable 
under an open subgroup of $I_\eta$. Thus the stabilizer $H$ of $\cF_0$ is an 
open subgroup of $G_\eta$ containing $\sigma(G_s)$. Write 
$G_\eta/H=\{g_1H,\dots,g_mH\}$ with $g_i\in I_\eta$. Then 
$\sum_{i}\rho_{g_i}(\cF_0)$ is stable under $\rho$. Thus it suffices to prove 
the analogue of (c)$\implies$(a) for torsion-free $\Z_\ell$-perverse sheaves. 
This case reduces immediately to the case of $\Z/\ell^n$-perverse sheaves. In 
the latter case, the action of $I_\eta$ factorizes through a finite quotient 
by an open subgroup. Thus, up to replacing $\eta$ by a finite extension, we 
may assume that $I_\eta$ acts trivially. In this case, the assertion is 
trivial. 
\end{proof}

With the help of Lemma \ref{l.descent}, we can modify the monodromy operator 
of a perverse sheaf on $Y\times_s \eta$ using the same formula as in the 
theory of $\ell$-adic representations of $G_\eta$. 

\begin{construction}\label{c.twist}
Let $Y$ be a scheme of finite type over $\F_q$. Let $\cP\in \Perv_c(Y\times_s 
\eta,\Q_\ell)$ and let $N\colon \cP(1)\to \cP$ be a nilpotent operator. We 
construct $\cP\langle N\rangle \in \Perv_c(Y\times_s \eta,\Q_\ell)$ 
satisfying $\pi_Y^*(\cP\langle N\rangle) \simeq \pi_Y^*\cP$, as follows. Let 
$(\pi_Y^*\cP,\rho)$ be the image of $\cP$ under the functor 
\eqref{e.descent}. For any continuous section $\sigma\colon G_s\to G_\eta$ of 
$r$, we define an action $\rho\langle N \rangle_\sigma$ of $G_\eta$ on 
$\pi_Y^*\cP$ by the formula 
\[\rho\langle N\rangle_\sigma (h\sigma(g))=\exp(Nt_\ell(h))\rho(h\sigma(g))\]
for $g\in G_s$, $h\in I_\eta$, where $t_\ell\colon I_K\to \Z_\ell(1)$ denotes 
the projection. Then $(\pi_Y^*\cP,\rho\langle N\rangle_\sigma)$ satisfies 
condition (c) of Lemma \ref{l.descent}. Indeed, the continuity of 
$\rho\langle N\rangle_\sigma |_{I_\eta}$ is clear and  $\sigma^*(\rho\langle 
N\rangle_\sigma)=\sigma^*\rho$. If $\sigma'\colon G_s\to G_\eta$ is another 
section of $r$ satisfying $\sigma'(\Fr_q)=h_0\sigma(\Fr_q)$, then we have an 
isomorphism 
\[\exp\left(\frac{1}{q^{-1}-1}Nt_\ell(h_0)\right)\colon (\pi_Y^*\cP,\rho\langle 
N\rangle_{\sigma})\simto (\pi_Y^*\cP,\rho\langle N\rangle_{\sigma'}).
\] 
Thus $(\pi_Y^*\cP,\rho\langle N \rangle_\sigma)$ does not depend on $\sigma$ 
up to isomorphism. By Lemma \ref{l.descent}, $(\pi_Y^*\cP,\rho\langle N 
\rangle_\sigma)$ descends to an object of $ \Perv_c(Y\times_s \eta,\Q_\ell)$, 
which we denote by $\cP\langle N \rangle $. 

Let $N_\cP$ denote the monodromy operator of $\cP$. Then $N_\cP$ and $N$ 
commute with each other and both descend to $\cP\langle N \rangle$. We have 
$N_{\cP\langle N\rangle}=N_\cP+N$. 
\end{construction}

We can now prove Theorem \ref{t.ex}. Let us recall the statement, which holds 
in fact without assuming $\chara(K)=0$. 

\begin{theorem}
For every projective curve $Y$ over $\F_q$, there exists a monodromy-pure 
perverse sheaf $\cP\in \Perv_c(Y\times_s \eta,\Q_\ell)$ such that 
$R(a\times_s \eta)_*\cP$ is not monodromy-pure. Here $a\colon Y\to s$ is the 
structure morphism. 
\end{theorem}

\begin{proof}
Note first that finite pushforward preserves monodromy-pure perverse sheaves. 
See for example \cite[Lemma 2.7.5]{HZ}. Thus, up to replacing $\F_q$ by a 
finite extension and $Y$ by a connected component of its normalization, we 
may assume that $Y$ is connected smooth and has at least two rational points, 
$P_1\neq P_2$. Let $i\colon V=\{P_1,P_2\}\to Y$ be the closed immersion and 
$j\colon Y\backslash V\to Y$ the complementary open immersion. The short 
exact sequence 
\[0\to j_!\Q_\ell \to \Q_{\ell,Y} \to i_* \Q_\ell \to 0\] 
of sheaves on $Y$ induces a short exact sequence
\[0\to i_*\Q_\ell \to j_!\Q_\ell[1] \to \Q_{\ell,Y}[1]\to 0\]
of perverse sheaves on $Y$. Since $\Ext^k_Y(i_*\Q_\ell(-1),i_*\Q_\ell)=0$ for 
all $k\in \Z$, we have  
\[\Ext^1_Y(i_*\Q_\ell(-1),j_!\Q_\ell[1])\simeq 
\Ext^1_Y(i_*\Q_\ell(-1),\Q_{\ell,Y}[1])\simeq H^2_V(Y,\Q_\ell(1)).
\] 
Thus the cycle class $\cl([V])\in H^2_V(Y,\Q_\ell(1))$ defines a perverse 
sheaf $\cG$ on $Y$, extension of $i_*\Q_\ell(-1)$ by $j_!\Q_\ell[1]$. Then 
$\cG$ is equipped with an increasing filtration $F$ with graded pieces 
\[\gr^F_{-1}\cG=i_*\Q_\ell,\quad \gr^F_0\cG=\Q_{\ell,Y}[1],\quad \gr^F_{1}\cG=i_*\Q_\ell(-1),\]
and $\gr^F_k\cG=0$ for $\lvert k\rvert>1$. Note that $\gr^F_k\cG$ is pure of 
weight $k+1$ for all $k\in \Z$. 

Now consider the map
\[N\colon \cG(1)\to \gr^F_1\cG(1) 
\xrightarrow{\phi}\gr^F_{-1}\cG\to \cG,
\] 
where $\phi\in \End(i_*\Q_\ell)\simeq \End(i_{1*}\Q_\ell) \oplus 
\End(i_{2*}\Q_\ell)$ is given by 
$(-\id_{i_{1*}\Q_\ell},\id_{i_{2*}\Q_\ell})$. Here $i_k\colon \{P_k\}\to Y$ 
is the inclusion for $k=1,2$. We have $N^2=0$. Let $p_Y\colon Y\times_s \eta 
\to Y$ be the projection. Applying Construction \ref{c.twist} to 
$p_Y^*N\colon p_Y^*\cG(1)\to p_Y^*\cG$, which we still denote by $N$, we get 
$\cP=(p_Y^*\cG)\langle N\rangle\in \Perv_c(Y\times_s \eta,\Q_\ell)$. We have 
$N_\cP=N$ and the monodromy filtration $M$ on $\cP$ coincides with the 
filtration $F$ on $\cG$ when restricted to $\pi_Y^*\cP\simeq \cG|_{Y_{\bar 
s}}$. Moreover, $\gr^M_j\cP\simeq p_Y^*\gr^F_j\cG$ for all $j\in \Z$. In 
particular, $\cP$ is monodromy-pure of weight $1$. Let $a\colon Y\to s$ be 
the structure morphism. In the monodromy spectral sequence 
\[E_1^{ij}=R^{i+j}(a\times_s \eta)_* \gr_{-i}^M \cP\Rightarrow R^{i+j}(a\times_s \eta)_*\cP,\]
$d_1^{-1,1}\colon E_1^{-1,1}\to E_1^{0,1}$ can be identified with the 
summation map $\Q_\ell(-1)\oplus\Q_\ell(-1)\to \Q_\ell(-1)$ and 
$d_1^{0,-1}\colon E_1^{0,-1}\to E_1^{1,-1}$ can be identified with the 
diagonal map $\Q_\ell\to \Q_\ell\oplus \Q_\ell$. The monodromy operator 
$N\colon E_1^{-1,1}(1)\to E_1^{1,-1}$ is given by the matrix 
$\begin{pmatrix}-1&0\\0&1\end{pmatrix}$. It follows that $N\colon 
E_2^{-1,1}(1)\to E_2^{1,-1}$ is the zero map and the monodromy operator on 
$R^{0}(a\times_s \eta)_*\cP$ is zero. However, $R^{0}(a\times_s \eta)_*\cP$ 
has $E_2^{1,-1}\simeq \Q_\ell$ as a subobject and $E_2^{-1,1}\simeq 
\Q_\ell(-1)$ as a quotient. Therefore, $R^{0}(a\times_s \eta)_*\cP$ is not 
monodromy-pure.  
\end{proof}

We sketch another counterexample to Question \ref{q.monodromy}, obtained by 
modifying the monodromy operator of nearby cycles. 

\begin{example}
Let $X$ be a strictly semistable scheme over $\Spec(\cO_K)$, projective of 
relative dimension $1$. Assume that the special fiber $Y=X_{s}$ consists of 
$n\ge 2$ components meeting at $n$ rational points $P_1,\dots,P_n$ such that 
the dual graph is an $n$-cycle. The case $n=2$ is shown in the picture below. 
\begin{center}
\begin{tikzpicture}
    % y = 0
    \draw (-1.4,0) -- (1.4,0);
    
    % y = x^2/2 - 1/2
    \draw[domain=-1.4:1.4,samples=100] plot (\x,{0.5*\x*\x - 0.5});
    
    % intersection
    \coordinate (A) at (-1,0);
    \coordinate (B) at (1,0);
    \fill[black] (A) circle (1pt) node[below] {$P_1$};
    \fill[black] (B) circle (1pt) node[below] {$P_2$};
    
\end{tikzpicture}
\end{center}
Let $\cP=R\Psi^\alg_X \Q_\ell[1]\in \Perv_c(Y\times_s \eta,\Q_\ell)$. Then 
$\cP$ is monodromy-pure of weight $1$, with graded pieces of the monodromy 
filtration $M$ given by 
\[\gr^M_{-1}\cP=i_*\Q_\ell,\quad \gr^M_0\cP=f_*\Q_{\ell}[1],\quad \gr^M_{1}\cP=i_*\Q_\ell(-1),\]
and $\gr^M_k\cP=0$ for $\lvert k\rvert>1$. Here $i\colon \{P_1,\dots,P_n\}\to 
Y$ is the closed immersion, $f\colon \tilde Y\to Y$ is the normalization of 
$Y$, and we dropped $p_Y^*$ from the notation, where $p_Y\colon Y\times_s 
\eta\to Y$ denotes the projection. 

Choose $\alpha_1,\dots,\alpha_n\in \Q_\ell^\times$ such that 
$\alpha_1+\dots+\alpha_n=0$. Consider the map 
\[N'\colon \cP(1)\to \gr^M_1\cP(1) 
\xrightarrow{\phi}\gr^M_{-1}\cP\to \cP,
\] 
where $\phi\in \End(i_*\Q_\ell)\simeq \bigoplus_{k=1}^n\End(i_{k*}\Q_\ell)$ 
is given by $(\alpha_k\id_{i_{k*}\Q_\ell})_k$. Here $i_k\colon \{P_k\}\to Y$ 
is the inclusion for $1\le k\le n$. Let $\cP'=\cP\langle N'-N_\cP\rangle\in 
\Perv_c(Y\times_s \eta,\Q_\ell)$ (Construction \ref{c.twist}). Then 
$N_{\cP'}=N'$ and the monodromy filtrations on $\cP'$ and $\cP$, both denoted 
by $M$, coincide on $\pi_Y^*\cP$. Moreover,  $\gr^{M}_k \cP'\simeq \gr^M_k 
\cP$ for all $k\in \Z$. In particular, $\cP'$ is monodromy-pure of weight 
$1$. Let $a\colon Y\to s$ be the structure morphism. In the monodromy 
spectral sequence 
\[E_1^{ij}=R^{i+j}(a\times_s \eta)_* \gr_{-i}^{M} \cP'\Rightarrow R^{i+j}(a\times_s \eta)_*\cP',\]
$E_2^{-1,1}\subseteq E_1^{-1,1}$ can be identified with the diagonal 
embedding $\Q_\ell(-1)\subseteq\Q_\ell(-1)^n$ and the quotient map 
$E_1^{1,-1}\to E_2^{1,-1}$ can be identified with the summation map $ 
\Q_\ell^n\to \Q_\ell$. Since $N'\colon E_1^{-1,1}(1)\to E_1^{1,-1}$ is given 
by the diagonal matrix $\diag(\alpha_1,\dots,\alpha_n)$,  $N'\colon 
E_2^{-1,1}(1)\to E_2^{1,-1}$ is the zero map. Thus $R^{0}(a\times_s 
\eta)_*\cP'$ is not monodromy-pure.
\end{example}

\end{document}